\documentclass[12pt]{amsart}
\usepackage[latin1]{inputenc}
\usepackage{color}
\usepackage{bbm}
\usepackage{amsmath,amssymb}
\usepackage{enumerate}
\usepackage{mathrsfs}
\usepackage{verbatim}

\newtheorem{thm}{Theorem}[section]
\newtheorem{cor}[thm]{Corollary}

\theoremstyle{definition}

\newtheorem{rem}[thm]{Remark}

\numberwithin{equation}{section}

\makeatletter
\newcommand{\PP}{\mathbb{P}}

\newcommand{\dif}{\,\mathrm{d}}
\newcommand{\cP}{\mathcal{P}}

\newcommand{\rd}{\,\mathrm{d}}

\newcommand{\charfun}{\ensuremath{\mathbbm 1}}

\allowdisplaybreaks

\begin{document}
\title{Extremal Distributions of Discrepancy functions}
\author[R. Kritzinger]{Ralph Kritzinger}
\address{Institute of Financial Mathematics and Applied Number Theory, Johannes Kepler University Linz, Austria, 4040 Linz, Altenberger Strasse 69}
\email{ralph.kritzinger@jku.at}
\author[M. Passenbrunner]{Markus Passenbrunner}
\address{Institute of Analysis, Johannes Kepler University Linz, Austria, 4040 Linz, Alten\-berger Strasse 69}
\email{markus.passenbrunner@jku.at}
\keywords{discrepancy function, convolution inequalities}
\subjclass[2010]{11K06, 11K38, 42A85, 46E30}

\begin{abstract}
The irregularities of a distribution of $N$ points in the unit interval are often measured with various notions of discrepancy.
The discrepancy function can be defined with respect to intervals of the form $[0,t)\subset [0,1)$ or arbitrary subintervals of the unit interval.
	In the former case, it is a well known fact in discrepancy theory that the $N$-element point set in the unit interval with the lowest $L_2$ or
	$L_{\infty}$ norm of the discrepancy function is the centered regular
	grid
	$$ \Gamma_N:=\left\{\frac{2n+1}{2N}: n=0,1,\dots,N-1\right\}. $$
	We show a stronger result on the distribution of discrepancy functions of point sets in $[0,1]$, which
	basically says that the distribution of the discrepancy function of $\Gamma_N$ is in some sense minimal among all $N$-element
	point sets. As a consequence, we can extend the above result to
	rearrangement-invariant norms,
	including $L_p$, Orlicz and Lorentz norms. 

	We study the same problem for the discrepancy notions with respect to arbitrary subintervals. In this case, we will observe that
	we have to deal with integrals of convolutions of functions. To this end, we prove a general upper bound on such expressions, which
	might be of independent interest as well.
\end{abstract}

\maketitle

\section{Introduction}	
Let $\mathcal P=\{x_0,\ldots, x_{N-1}\}$ be an $N$-element point set in
	the unit interval $[0,1]$ which is always assumed to be arranged
	increasingly.
	Denote by $D=D_{\mathcal P}$ its one-parameter discrepancy function
	\begin{equation}\label{eq:defD}
	D(t) = \sum_{n=0}^{N-1} \charfun_{[0,t)}(x_n) - Nt,
	\end{equation}
	where $\charfun_{A}$ denotes the indicator function of the set $A$.
	We also consider a two-parameter discrepancy function $\widetilde{D} =
	\widetilde{D}_{\mathcal P}$, defined by
\begin{equation}\label{eq:defDtilde}
	\widetilde{D}(t_1,t_2) =\sum_{n=0}^{N-1}
\charfun_{[t_1,t_2)}(x_n)-N(t_2-t_1),\qquad 0\leq t_1\leq t_2\leq 1.
\end{equation}
Therefore, the discrepancy functions measure the deviation of the actual number of points in a subinterval of $[0,1]$ (the so-called test sets) and the expected number of points under
the assumption of uniform distribution. This deviation is measured with respect to intervals anchored in the origin in the case of $D_{\cP}$ and with respect to
arbitrary subintervals of $[0,1]$ in the case of $\widetilde{D}_{\mathcal{P}}$.

  One usually considers a norm of the discrepancy function as a quantitative measure of the irregularities of distribution
	of a point set. The best-studied cases are those of the $L_p$ norms for $1\leq p \leq \infty$, where we speak of $L_p$ discrepancy
	for finite $p$ and of star discrepancy for $p=\infty$ in case of the
	one-parameter discrepancy function. Note that for a measurable function $f: A\to \mathbb{R}$ defined on a domain $A\subseteq
	\mathbb{R}^d$ with $|A|>0$, where $|\cdot|$ denotes the $d$-dimensional Lebesgue measure, we define the $L_p$ (quasi-) norm of $f$ for $p\in(0,\infty)$ by
	$$ \|f\|_p:=\left(\int_A |f(t)|^p \dif t\right)^{1/p} $$
and the $L_\infty$ norm by
$$ \|f\|_{\infty}:= \inf \{ \lambda \geq 0 : |f|\leq \lambda \text{ a.e.}\}. $$
If we take the same norms of the two-parameter discrepancy function, one usually speaks of extreme $L_p$ and star discrepancy, respectively. Consult e.g.~\cite{Doerr2014} for an overview on these notions and~\cite{Matousek1999} for an excellent introduction to discrepancy theory. The smaller these discrepancy notions of a point set $\mathcal{P}$, the more uniformly
	it is distributed in the unit interval (see
	e.g.~\cite{KuipersNiederreiter1974}). The determination of those
	$N$-element point sets that have minimal discrepancy is a very difficult and largely 
	unsolved problem in dimensions higher than one. We refer to~\cite{White1977} for the $N$-element
	point sets in $[0,1]^2$ with minimal star discrepancy for $N=1,\dots,6$
	and to~\cite{Pillards2006} and \cite{LarcherPillichshammer2007} for the one
	and two-element point sets in the $d$-dimensional unit cube $[0,1]^d$
	with minimal $L_2$, star and extreme star discrepancy, respectively.
	Moreover, a systematic search for the $N$-element point sets with
	minimal $L_2$ discrepancy, measured with respect to periodic boxes, up to $N=16$,
	was performed in \cite{HinrichsOettershagen2016}.
	However, for point sets in the one-dimensional unit interval $[0,1]$ the answer is known for the star and $L_2$ discrepancy and every natural number $N$.
	By Niederreiter (see \cite[Corollary 1.1]{Niederreiter1973} and
	\cite[Theorem 2.6]{Niederreiter1992}) we have for the $L_2$ and star discrepancy of a set of points $\mathcal{P}=\{x_0,\dots,x_{N-1}\}$ the explicit formulas
	\[ \|D_{\mathcal{P}}\|_2=N\sum_{n=0}^{N-1}\left(x_n-\frac{2n+1}{2N}\right)^2+\frac{1}{12} \] and
	\[ \|D_{\mathcal{P}}\|_{\infty}=N\max_{0\leq n\leq N-1}\left|x_n-\frac{2n+1}{2N}\right|+\frac{1}{2}, \] 
	respectively. As an immediate consequence we find that the centered regular grid
\[
	\Gamma_N := \Big\{\frac{2n+1}{2N} : n=0,\ldots, N-1\Big\}.
\]
is the unique minimizer of both the $L_2$ and star discrepancy among all $N$-element points in $[0,1]$. 

The situation is similar for the extreme discrepancy notions. Niederreiter
\cite[Theorem 2.7]{Niederreiter1992} was able to show the explicit formula
$$   \|\widetilde{D}_{\mathcal{P}}\|_{\infty}=1+N\max_{0\leq n\leq N-1} \left(\frac{n}{N}-x_n\right)-N\min_{0\leq n\leq N-1} \left(\frac{n}{N}-x_n\right)$$
for $\mathcal P=\{x_0,\ldots x_{N-1}\}\subset [0,1]$.
Furthermore, it is not hard to prove the following formula for its extreme $L_2$ discrepancy.
 By a straight-forward computation of the integrals in its definition and some elementary algebra we find
 \begin{equation}\label{eq:Dtilde2}
	   \|\widetilde{D}_{\mathcal{P}}\|_{2}^2
			 = \frac{1}{12}+\frac12 \sum_{n,m=0}^{N-1}\left(x_n-x_m-\frac{n-m}{N}\right)^2.
	\end{equation}
 Therefore, the only minimizing point sets with $N$ elements of the extreme star and $L_2$ discrepancy are translated regular grids of the form
\begin{equation} \label{transl}
	\Gamma_N^{\delta} = \Big\{\frac{n}{N}+\delta : n=0,\ldots, N-1\Big\}
	\text{\, for some $\delta\in\Big[0,\frac{1}{N}\Big)$}.
\end{equation}
Observe that with this notation, we have $\Gamma_N^{1/(2N)}
= \Gamma_N$.

The question arises whether these statements
remain true if we take other norms of the one- and two-parameter discrepancy function of $\mathcal{P}$. To this end, we will show results on the distribution
of the discrepancy functions, which is motivated by the fact that the $L_p$ norm
and various other norms of the discrepancy functions are determined by the
distribution of their absolute values. In general, by distribution  we mean the
following: Let $f: A\to \mathbb{R}$ be a measurable function on a domain with $A\subseteq
\mathbb{R}^d$ and $|A|>0$. 
Then we define $\PP_A(f<\alpha):=|\{t\in A: |f(t)| <
\alpha\}|/|A|$. Here, we denote by $|\cdot|$ the Lebesgue measure on $\mathbb R^d$. We usually suppress the lower index in $\PP_A$. Hence, we ask for results on the distributions $\PP(|D_{\cP}|<\alpha)$ and $\PP(|\widetilde{D}_{\cP}|<\alpha)$ of the one- and two-parameter discrepancy function, respectively. To be more precise, we will show that for any natural number $N$ and any point sequence $\mathcal{P}$ with $N$ elements, we have
	\[
		\mathbb P(|D_{\mathcal{P}}| < \alpha) \leq \mathbb
		P(|D_{\Gamma_N}|<\alpha) \text{\quad  and \quad} \mathbb P(|\widetilde{D}_{\mathcal{P}}| < \alpha) \leq \mathbb
		P(|\widetilde{D}_{\Gamma_N}|<\alpha), \qquad \alpha>0,
	\]
with equality for all $\alpha>0$ only if $\cP$ is the centered regular grid
$\Gamma_N$ in the first inequality or a translation thereof in the second. Normed function spaces where the norm of a function is determined by its distribution are called 
rearrangement-invariant and include Orlicz and Lorentz spaces, for instance, which both generalize the $L_p$ spaces.
As a general reference to those notions, we use \cite{BennettSharpley1988}.
We will use our results on the distribution functions of $D$ and $\widetilde{D}$
to identify the centered regular grid $\Gamma_N$ and its translations as the 
(only) minimizers of each such norm of $D$ and $\widetilde{D}$, respectively. 

\section{The distribution of the one-parameter discrepancy function}\label{sec:one_parameter}
The function $D$ and its extremal distribution are easily analyzed.
To this end, we investigate the distribution function
  $ \PP(D\leq \alpha):=|\{t\in [0,1]: D(t)\leq \alpha\}|$
	of $D=D_{\cP}$. 
	We observe that if
$\ell$ is a linear function on an interval $I\subset \mathbb R$ of finite length
that has a slope of $k\neq 0$, we get
\[
	|\{ x\in I: \ell(x)\in (a,b) \}| = \int_a^b
	\frac{\charfun_{\ell(I)}(t)}{|k|}  \dif t,\qquad a<b.
\]
Let now $g$ denote the density of $D$; i.e. $g\geq 0$ with $\int g=1$  is such that
\[
	\mathbb P(D \in (a,b)) =\int_{a}^{b} g(t)\rd t,\qquad a<b.
\]
Observe that $D$ consists of $N+1$ linear pieces with slope $-N$ on the
intervals $(x_{n-1},x_n)$  for $n=0,\ldots,N$ with $x_{-1}=0, x_N=1$.
Therefore, by the above argument and setting $I_n = (n-Nx_n, n-Nx_{n-1})$ for
$n=0,\ldots,N$, 
\begin{equation}\label{eq:form_of_g}
	g = \frac{1}{N} \sum_{n=0}^{N} \charfun_{I_n}.
\end{equation}
Since $I_0\cap I_{N}=\emptyset$, 
the function $g$ is piecewise constant with $g\in \{ j/N : 0\leq j \leq N \}$.
For instance, the density $g$ corresponding to the translated grid
$\Gamma_N^{\delta}$---as defined
in~\eqref{transl}---is given by $\charfun_{[-N\delta,1-N\delta)}$. 

	It is easily seen that the properties $0\leq g\leq 1$ and $\int g=1$
	imply that for intervals $I$ symmetric around $0$,  we have the inequality
	\begin{equation}\label{eq:charest}
		\int_I g(t)\dif t \leq \int_I M_1(t)\dif t
	\end{equation}
	with $M_1 = \charfun_{(-1/2,1/2)}$.
	In \eqref{eq:charest}, equality  for all symmetric intervals $I$ around
	$0$ holds exactly if $g=M_1$ a.e.
	This immediately implies the following result:

\begin{thm} \label{theo1}
	For any natural number $N$ and any point sequence $\mathcal{P}$ with $N$ elements, we have
	\[
		\mathbb P(|D_{\mathcal{P}}| < \alpha) \leq \mathbb
		P(|D_{\Gamma_N}|<\alpha), \qquad \alpha>0,
	\]
	and equality for all $\alpha>0$ holds if and only if $\mathcal{P}=\Gamma_N$.
\end{thm}
\begin{proof}
With the density $g$ of $D_{\mathcal{P}}$ introduced above we can write for
$\alpha>0$
\[
	\mathbb P ( |D_{\mathcal{P}}| < \alpha ) = \mathbb P ( D_{\mathcal{P}} \in
	(-\alpha,\alpha)) = \int_{-\alpha}^\alpha g(t)\dif t \leq
	\int_{-\alpha}^\alpha M_1(t)\dif t = \mathbb
	P(|D_{\Gamma_N}|<\alpha)
\]
by \eqref{eq:charest} with equality for all $\alpha>0$ exactly for
$g=M_1$ a.e., which is the case only if $\mathcal P = \Gamma_N$.
\end{proof}

\section{Convolution inequalities} \label{prelim}

In this section, we show a general inequality involving integrals and the convolution
of functions with certain properties, which will lead to estimates for the
distribution function of $\widetilde{D}$. 
Before we state these results, we explain 
the relationship between $\widetilde{D}$ and convolutions of functions.
The function $\widetilde{D}$ in \eqref{eq:defDtilde} is only defined on the set $S=\{(t_1,t_2)\in
[0,1]^2 : t_1\leq t_2\}$ and on that set, it can be written as
$\widetilde{D}(t_1,t_2) = D(t_2) - D(t_1)$ with $D$ defined as in
	\eqref{eq:defD}. We extend $\widetilde{D}$ to the set
$[0,1]^2\setminus S$ by the same formula, which implies that $|\widetilde{D}|$
has the same distribution on the set $S$ as it has on $[0,1]^2$. In the
following, we always consider $\widetilde{D}$ to be defined on $[0,1]^2$.
 Denoting again by $g$ the probability density of $D$, the
 probability density of $-D$ is given by $\widetilde{g}$ with
 $\widetilde{g}(x):=g(-x)$ for $x\in\mathbb R$.
 Since the probability density of the sum of independent random variables is the convolution
 of their densities, we 
 obtain
 \begin{equation}\label{eq:PDtilde}
\mathbb P( |\widetilde{D}| < \alpha) = 
|\{(t_1,t_2)\in [0,1]^2: D(t_2) - D(t_1) \in
	(-\alpha,\alpha)\}| = \int_{-\alpha}^\alpha (g * \widetilde{g})(t)\dif
	t,
\end{equation}
 with the convolution $f*g$ given by $(f*g)(x) = \int f(x-y)g(y)\dif y$.
 Therefore, we have to work with integrals where the integrands involve convolutions of functions
with certain properties like those of $g$ stated above. In the rest of this section, we will
prove an upper bound on such integrals.

Let $f:\mathbb R\to \mathbb [0,\infty]$ be a non-negative function. We say that $f$ is \emph{symmetrically
decreasing} (or short: s.d.), if $f(-x)=f(x)$ for all $x\in \mathbb R$ and, for all $0<x<y$, we
have $f(y)\leq f(x)$. 
	It is easy to see that each s.d. function
	$f$ can be
approximated from below pointwise a.e. by an increasing sequence of simple functions of the form
$\sum_i c_i \charfun_{(-t_i,t_i)}$
for $t_i\in[0,\infty]$ and $c_i\geq 0$.
Characteristic functions of the form $\charfun_{I}$ for a symmetric interval $I$
around zero are
s.d. and the convolution of two such functions is given by 
\[
	\charfun_{I} * \charfun_{J} (x) = |I \cap (x+J)|,\qquad x\in\mathbb R,
\]
which is again s.d. 
Therefore, by pointwise approximation with simple functions
and the monotone convergence theorem, we conclude that the convolution $f*g$ of two
arbitrary s.d. functions $f,g$ is again s.d.
Now we show that the function $M_1=\charfun_{(-1/2,1/2)}$ is largest 
among all s.d. functions $g$ satisfying $0\leq g\leq 1$ and $\int g=1$ in the
following sense:

\begin{thm}\label{thm:main}
	Let $f,g,h$ be symmetrically decreasing
	functions on $\mathbb R$ with
	$0\leq g\leq 1$ and $\int g=1$. Then,
	\begin{equation}\label{eq:M1}
		\int h(x) (f *  g)(x)\dif x 
		\leq \int h(x) (f * M_1)(x)\dif x.
	\end{equation}

	Moreover, if $f$ is not constant a.e. and $f*g\in L_1$, equality for all $h$ here implies
	that $g=M_1$ (a.e).
 \end{thm}
 \begin{proof}
	We begin by proving \eqref{eq:M1}.
	Approximating the s.d. function $h$ by simple
	functions as described above and using the monotone convergence theorem, it suffices to
	consider $h=\charfun_{(-\alpha,\alpha)}$ with $\alpha>0$.
	 Observe that by Fubini's theorem
	 \[
		 \int_{-\alpha}^\alpha (f *  g)(t)\dif t =  \int
		f(z)
		\int_{-\alpha}^\alpha g(t-z)\dif t\dif z.
	 \]	
	 Let 
	 $v(z)=v_\alpha(z)= (\charfun_{(-\alpha,\alpha)}*g)(z)=\int_{-\alpha}^\alpha g(t-z) \dif t =
	 \int_{-\alpha-z}^{\alpha-z} g(t)\dif t $. Note that $v$, as the
	 convolution of two s.d. functions, is s.d.
	 We next show the following properties of $v$:
	 \begin{enumerate}
		 \item $0\leq v \leq \min(2\alpha,1)$,
		 \item $\int v = 2\alpha$,
		 \item $v$ is $1$-Lipschitz.
	 \end{enumerate}
	 Since $0\leq g\leq 1$ and $\int g=1$, we have  
	 $0\leq v\leq
	 \min(2\alpha,1)$, showing property (1).
	 For property (2), we just calculate  
	 \[
		 \int v(z)\dif z = \int_{-\alpha}^\alpha
		 \int g(t-z)\dif z \dif t = 2\alpha,
	 \]
	 where in the last equality, we used $\int g=1$.
	 Finally, property (3) is seen by the fact that
	 $v(z)=\int_{-\alpha-z}^{\alpha-z} g(t)\dif t$ implies that, for $\rho>0$, 
	 $v(z+\rho)-v(z)$ can be written as 
	 \[
		v(z+\rho) - v(z) = \int_I g(t)\dif t - \int_J g(t)\dif t
	 \]
	 for two intervals $I,J$ with $|I|\leq \rho$ and $|J|\leq \rho$. Therefore,
	 the assumption $0\leq g \leq 1$ implies $|v(z+\rho)-v(z)| \leq \rho$, which
	 is (3).

	 Next, define  $v_0= v_{0,\alpha}=
	\charfun_{(-\alpha,\alpha)} * M_1$. Then, $v_0$ is the s.d.
	function that equals $\min(2\alpha,1)$ on $[0, |\alpha - 1/2|]$, has
	slope $-1$ on $(|\alpha-1/2|,\alpha+1/2)$ and equals zero on
	$[\alpha+1/2,\infty)$.
	By the properties (1) and (2) of $v$ and the definition of $v_0$,
	for $t_0>0$ with $t_0<|\alpha-1/2|$ or $t_0>
	\alpha+1/2$, we clearly have 
	\begin{equation}\label{eq:t_0}
		\int_{-t_0}^{t_0} v(z)\dif z \leq \int_{-t_0}^{t_0} v_0(z)\dif
		z.
	\end{equation}
	If $|\alpha-1/2| \leq t_0\leq \alpha+1/2$, we distinguish the cases
	$v(t_0) \leq v_0(t_0)$ and $v(t_0) > v_0(t_0)$.
	In the first case, by (3), (1) for $v$ and the definition
	of $v_0$, we have \eqref{eq:t_0}. In the second case, by the same
	reasoning, $\int_{t_0}^\infty v(z)\dif z \geq \int_{t_0}^\infty
	v_0(z)\dif z$, which, using property (2) for $v$ and $v_0$ also yields
	\eqref{eq:t_0}. By approximating the s.d.
	function $f$ as above by an increasing sequence of simple functions
	and using the monotone convergence theorem,
	inequality \eqref{eq:t_0} implies
	\begin{equation}\label{eq:M1_vorstufe}
		\int_{-\alpha}^\alpha (f*g)(t)\dif t = \int f(z) v(z)\dif z \leq \int f(z) v_0(z)\dif z
		= \int_{-\alpha}^\alpha (f*M_1)(z)\dif z,
	\end{equation}
	which concludes the proof of inequality \eqref{eq:M1}.

	Now, we prove the equality part.
	We assume that $f$ is not constant a.e. and $g\neq M_1$.
	This implies that $a:= \operatorname{ess\,inf} f <
	\operatorname{ess\,sup} f =:b$ and the existence of
	$x_0\in (0,1/2)$ 
	so that $g(t)< 1$ for all $t\in
	(x_0,1/2)$. Set  $\delta = 1/2-x_0>0$.
	
	First, we show that for $\alpha\geq 0$ and $I_\alpha=(\alpha-1/2,
	\alpha-x_0)$, we have the inequality
	\begin{equation}\label{eq:v_v0}
		v_\alpha(z) <  v_{0,\alpha}(z),\qquad z\in I_\alpha\cup
		(-I_\alpha)
	\end{equation}
	Indeed, for $z\in I_\alpha$, we decompose the set
	$(-\alpha-z,\alpha-z)$ into the two intervals
 $J_1 :=
	(\min(-\alpha-z,x_0), x_0)$ and  $J_2 := (\max(x_0,-\alpha-z),\alpha-z) \subset
	(x_0,1/2)$, with $|J_2|>0$.
	By the properties of $g$ and $M_1$, the inequality $\int_{J_1} g(t)\dif
	t\leq \int_{J_1} M_1(t)\dif t$ is true. Moreover, for $t\in J_2$, we
	have $g(t) < 1 = M_1(t)$. The fact that $|J_2|>0$
	then implies 
	$\int_{J_2} g(t)\dif t < \int_{J_2} M_1(t)\dif t$. 
	Since $v(z)= \int_{J_1} g(t)\dif t + \int_{J_2} g(t)\dif t$ and $v_0(z)= 
	\int_{J_1} M_1(t)\dif t + \int_{J_2} M_1(t)\dif t$, adding up the
	above two inequalities yields \eqref{eq:v_v0} for $z\in I_\alpha$. Since the
	functions 
	$v$ and $v_0$ are both s.d., \eqref{eq:v_v0} also holds for $z\in -I_\alpha$.

	Next, we choose the parameters $s_1,s_2$ with $a<s_1 < s_2<b$, $\eta\in(0,\delta/2)$ and $t_0>0$ in such a
	way that 
	\begin{itemize}
		\item $t_0-\eta>0$,
		\item $f(t)\leq s_1$ if $t\geq t_0+\eta$,
		\item $f(t)\geq s_2$ if $0\leq t\leq t_0$.
	\end{itemize}
	This is possible because, due to the continuity of the Lebesgue
		measure, the measure of the set $\{ f\in (s_1,s_2) \}$ can be
	chosen arbitrarily small if $s_2-s_1$ is sufficiently small.
		
	Define $U= (t_0-\eta, t_0+\eta) \cup (-t_0-\eta,
	-t_0+\eta)$ and 
	decompose $f = f_1 + f_2$ with
	\[
		f_1 = \big( \min(f,s_2) - s_1 \big) \cdot \charfun_{\{f\geq
		s_1\}}\cdot \charfun_U, \qquad f_2 = f-f_1.
	\]
	Observe that $f_2$ is s.d. and $f_1\geq 0$. Since
	$|(t_0-\eta, t_0+\eta)| = 2\eta \leq \delta= |I_\alpha| $ for all $\alpha$, we can 
	choose $\alpha>0$
	so that $U\subseteq I_\alpha \cup (-I_\alpha)$. Doing so,
		\eqref{eq:M1_vorstufe} implies $\int f_2(z)v_\alpha(z)\dif z
		\leq \int f_2(z) v_{0,\alpha}(z)\dif z$ and 
		\eqref{eq:v_v0} implies $\int f_1(z) v_\alpha(z)\dif z < \int
		f_1(z) v_{0,\alpha}(z)\dif z$, where in the strict inequality,
		we also use the fact that $f_1= s_2-s_1>0$ on a subset of
		$U$ having positive Lebesgue measure.
	Adding up those inequalities, we conclude (using that $f*g\in L_1$)
	\[
		\int_{-\alpha}^\alpha (f*g)(t)\dif t = \int f(z) v_\alpha(z)\dif z < \int f(z) v_{0,\alpha}(z)\dif z
		=\int_{-\alpha}^\alpha (f*M_1)(t)\dif t
	\]
	for $\alpha>0$ chosen above, which finishes the proof of the
	equality part.
 \end{proof}

 We next extend the result in Theorem~\ref{thm:main} to more general
 functions and arbitrarily many convolution factors. To this end, 
we need an important classical inequality involving convolutions and
symmetric rearrangements, which we now describe.
Two functions $f,g:\mathbb R\to \mathbb
[0,\infty]$ are \emph{equimeasurable}, if the level sets $ \{ f \geq \lambda \}$
	and $\{ g\geq\lambda \}$ for $\lambda\geq 0$ have the same Lebesgue measure.
The \emph{symmetric decreasing rearrangement} of a
non-negative function $f$ is given by the s.d. function
$f^*$ that is equimeasurable with $f$. 
The symmetric decreasing rearrangement is uniquely
determined up to equality almost everywhere.
The following inequality is 
 known for $n=1$ as the Hardy-Littlewood inequality, for $n=2$ as the Riesz
inequality and for $n>2$ it is due to Luttinger and Friedberg
\cite{FriedbergLuttinger1976}
\begin{align}\label{eq:riesz}
	\int f(x)(g_1 * \cdots * g_n)(x)\dif x &\leq \int f^*(x) (g_1^* * \cdots
	* g_n^*)(x)\dif x,
\end{align}
where $f,g_1,\ldots,g_n$ are arbitrary non-negative functions. For an even more
general version of this inequality, we refer to
\cite{BrascampLiebLuttinger1974}.
Symmetrically decreasing
rearrangements of functions, as well as the above inequality for $n=1,2$ are
treated in the classical book Inequalities \cite{HardyLittlewoodPolya1952} by Hardy,
Littlewood and P\'{o}lya.
We now define iterative convolutions of the function $M_1
=\charfun_{(-1/2,1/2)}$ as follows:
\[
	M_n = M_{n-1} * M_1,\qquad n\geq 2.
\]
The function $M_n$ is known as the centered cardinal
B-spline of order $n$. 

Combining Theorem \ref{thm:main} with inequality \eqref{eq:riesz} yields 
the following convolution inequality which will be used
 later to estimate the measure of level sets of the 
 extreme discrepancy $\widetilde{D}$ and may be of independent interest:
 \begin{thm} \label{nconv}
	Let $g_1,\ldots,g_n$ be functions on $\mathbb R$ so that $0\leq g_j\leq
	1$ and $\int g_j=1$ for all $1\leq j\leq n$. Then, we have for all
	non-negative functions $h$
		\begin{equation}\label{eq:cor}
			\int h(x) (g_1 * \cdots * g_n)(x)\dif x \leq \int h^*(x)
			M_n(x)\dif x.
		\end{equation}

	If equality holds for all s.d. functions $h$, it follows that $g_1^* =
	\cdots =g_n^* = M_1$ a.e.
 \end{thm}
 \begin{proof}	 
	 Applying  \eqref{eq:riesz} to the left hand side of \eqref{eq:cor}, we only have to
	 estimate the expression
	 \[
		\int h^*(x) (g_1^* * \cdots * g_n^*)(x)\dif x.
	 \]
	 If $n=1$, this is trivially estimated by $\int h^*(x) M_1(x)\dif x$. 
	 For $n\geq 2$,
	 we apply Theorem~\ref{thm:main} iteratively and use
	 that the convolution of s.d.
	 functions is again s.d. to deduce
	 \begin{equation}\label{eq:iterative}
		\int h^*(x) (g_1^* * \cdots * g_n^*)(x)\dif x\leq \int h^*(x) (M_1 * \cdots * M_1)(x)\dif x = \int h^*(x)
		M_n(x)\dif x.
	\end{equation}
	If equality holds in \eqref{eq:cor} (for all s.d. functions $h$) then, in particular, equality holds
	in \eqref{eq:iterative}. By the second part of Theorem \ref{thm:main},
this is the case if and only if $g_1^* = \cdots = g_n^* = M_1$ a.e.
 \end{proof}

\section{Main results and consequences} \label{main}

We now use the techniques of Section \ref{prelim} to 
show a similar result for the extreme discrepancy
$\widetilde{D}$ in Theorem \ref{theo2} below as we did for $D$ in Theorem~\ref{theo1}.
Then 
we calculate various minimal values of rearrangement invariant norms of
discrepancy functions in order to demonstrate how to apply the above theorems. 
The list of norms we consider here is by no means
exhaustive, but  Theorems~\ref{theo1} and \ref{theo2}
allow us to treat any desired rearrangement invariant norm.

 \begin{thm} \label{theo2}
	 We have for all $N$-element point sets $\mathcal{P}$,
	 \[
		 \mathbb P(|\widetilde{D}_{\mathcal{P}}| < \alpha) \leq \mathbb
		 P(|\widetilde{D}_{\Gamma_N}| < \alpha),\qquad \alpha>0,
	 \]
	 and equality for all $\alpha>0$ holds if and only if
	 $\mathcal{P}=\Gamma_N^\delta$ for some $\delta\in [0,1/N)$.
 \end{thm}

\begin{proof}
	By \eqref{eq:PDtilde}, we can write
   \[
\mathbb P( |\widetilde{D}| < \alpha)= \int_{-\alpha}^\alpha (g * \widetilde{g})(t)\dif t
 \]
 with $g$ being (in particular) a function with $0\leq g\leq 1$ and $\int g=1$.
 Now the claim follows from Theorem~\ref{nconv} for $n=2$ with $h=\charfun_{(-\alpha,\alpha)}$, $g_1=g$ and 
 $g_2=\widetilde{g}$, which yields
  $$ \mathbb P(|\widetilde{D}_{\mathcal{P}}| < \alpha)\leq \int_{-\alpha}^{\alpha} M_2(t)\rd t=\mathbb
		 P(|\widetilde{D}_{\Gamma_N}| < \alpha),\qquad \alpha>0. $$
	Equality (for all $\alpha>0$) in Theorem~\ref{nconv} holds only for
	$g^*=\widetilde{g}^*=M_1$.
	Recall that $g$ is of the form \eqref{eq:form_of_g}, i.e.,
	\[
		g = \frac{1}{N} \sum_{n=0}^{N} \charfun_{I_n}
	\]
	with $I_n = (n-Nx_n, n-Nx_{n-1})$ for $n=0,\ldots,N$ and $x_{-1}=0,
	x_N=1$. Therefore, the condition $g^* =M_1$ implies $I_1 = \cdots =
	I_{N-1} = I_0\cup I_{N}$, which gives
	$g=\charfun_{I_1}$. But this is only the case for the density of $D_{\Gamma_N^{\delta}}$
	for any $\delta\in [0,1/N)$.
\end{proof}

Let $\psi: [0,\infty)\to [0,\infty)$ be an absolutely continuous, strictly increasing function with $\psi(0)=0$. Then, $\psi$ is differentiable
	a.e. and $\psi(a) = \int_0^a \psi'(s)\dif s$ for all $a\geq 0$.
	Consider a function $f: A\to\mathbb{R}$, where $A\subseteq \mathbb{R}^d$
	with $|A|=1$. We define
	  $$ \|f\|_{\psi}:=\inf\left\{K>0: \int_A \psi\left(\frac{|f(t)|}{K}\right)\rd t\leq 1\right\} $$	
	with the usual convention $\inf\emptyset=\infty$. Note that $\|f\|_{\psi}$ matches the $L_p$ (quasi-)norm for $p\in (0,\infty)$ by choosing
	for $\psi$ the particular function $\psi_p: [0,\infty)\to [0,\infty), s\mapsto s^p$. For convex functions $\psi$ in general, $\|f\|_{\psi}$
	yields the Orlicz norm. However, the results in this paragraph hold for all functions $\psi$ with the less restrictive properties as stated above.
	Using Fubini's theorem, we perform the following short and well-known
	trick introducing the distribution function of $f$ to obtain
	\begin{align} \label{distorl}
	   \int_A \psi\left(\frac{|f(t)|}{K}\right)\dif t=& \int_A
	   \int_0^{|f(t)|/K}\psi'(\alpha)\dif \alpha \dif t 
		= \int_0^{\infty} \psi'(\alpha) \mathbb P(|f|\geq K\alpha)\dif\alpha.
	\end{align}
It is easy to see that for all $\alpha\geq 0$ we have
$\PP(|D_{\Gamma_N}| \geq \alpha)=\max\{0,1-2\alpha\}$ and $\PP(|\widetilde{D}_{\Gamma_N}|
\geq \alpha)=(1-\min\{\alpha,1\})^2$.
Define the functions $\Psi,T:[0,\infty)\to [0,\infty)$ such that
	$\Psi(0)=T(0)=0$ and 
$\Psi'\equiv \psi$, $T'\equiv \Psi$.
		Inserting $D$ and $\widetilde{D}$ instead of $f$ in \eqref{distorl}, 
Theorem~\ref{theo1} and Theorem~\ref{theo2} yield
	\begin{align*}
	   \int_0^1 \psi\left(\frac{|D_{\mathcal{P}}(t)|}{K}\right) \rd t \geq
	   \int_0^{1/(2K)} \psi'(\alpha) (1-2K\alpha)\dif\alpha=2K\, \Psi\left(\frac{1}{2K}\right)
	\end{align*}
	and
		\begin{align*}
	   \int_0^1 \int_{0}^1
	   \psi\left(\frac{|\widetilde{D}_{\mathcal{P}}(t_1,t_2)|}{K}\right) \rd
	   t_2 \rd t_1\geq  \int_0^{1/K} \psi'(\alpha)
	   (1-K\alpha)^2\dif\alpha=2K^2\, T\left(\frac{1}{K}\right)
	\end{align*}
	for every $N$-element point set $\cP$ in the unit interval, respectively. As a result, we obtain the following corollary.
	
	\begin{cor} \label{coro1}
	      Let $\psi: [0,\infty)\to [0,\infty)$ be an absolutely continuous, 
		      strictly increasing function with $\psi(s)=0$ and $\Psi,T$
		      as above.
		       Moreover, let $N$ be a non-negative integer.
	      
	      Then we have 
	\begin{equation} \label{inf1} \inf_{\# \mathcal{P}=N}\|D_{\mathcal{P}}\|_{\psi}=\inf\left\{K>0: 2K\, \Psi\left(\frac{1}{2K}\right)\leq 1\right\} \end{equation}
	and
	\begin{equation} \label{inf2} \inf_{\#
			\mathcal{P}=N}\|\widetilde{D}_{\mathcal{P}}\|_{\psi}=\inf\left\{K>0:
				2K^2\, T\left(\frac{1}{K}\right)\leq 1\right\}, \end{equation} 
	 where the infimum is extended over all $N$-element point sets in $[0,1]$.
	\end{cor}
	
	\begin{rem}
	 The special choice $\psi=\psi_p: s\mapsto s^p$ in~\eqref{inf1} and~\eqref{inf2} yields
				\begin{equation*} \inf_{\#
						\mathcal{P}=N}\|D_{\mathcal{P}}\|_{p}^p=
						\frac{1}{2^p(p+1)}\qquad
						\text{and}\qquad \inf_{\#
							\mathcal{P}=N}\|\widetilde{D}_{\mathcal{P}}\|_{p}^p=
							\frac{2}{(p+1)(p+2)}\end{equation*}
				for all $p\in (0,\infty)$. 
				This formula for $\widetilde D_{\mathcal P}$ with 
				$p=2$ and \eqref{eq:Dtilde2} for $\mathcal
				P=\Gamma_N$ are different by a
				factor of $2$, because we consider
				$\widetilde{D}_{\mathcal P}$ to be defined on $[0,1]^2$,
				whereas in formula \eqref{eq:Dtilde2} it is 
				integrated over the set $\{ (t_1, t_2)\in [0,1]^2
				: 0\leq t_1\leq t_2\leq 1\}$.
	\end{rem}

	Finally, we consider Lorentz norms. Let $f$ be a Lebesgue measurable
	function and $0<p,q<\infty$. We define the Lorentz norm
	$$ \|f\|_{p,q}:=p^{1/q}\left(\int_0^{\infty} \alpha^{q-1}\mathbb
	P(|f|\geq \alpha)^{p/q}\dif \alpha\right)^{1/q}. $$
	Theorem~\ref{theo1} and Theorem~\ref{theo2} then yield the following
	lower bounds on Lorentz norms.
	
		\begin{cor} \label{coro1:lorentz}
	      Let $B(x,y):=\int_0^1 t^{x-1}(1-t)^{y-1}\dif t$ for $x,y>0$ be the
	      Beta function, $0<p,q<\infty$ and $N$ a non-negative integer.

	     Then we have
	\begin{equation} \label{inf3} \inf_{\#
			\mathcal{P}=N}\|D_{\mathcal{P}}\|_{p,q}^q=\frac{p}{2^q}B(q,1+q/p) 
	\end{equation}
	and
	\begin{equation} \label{inf4} \inf_{\#
			\mathcal{P}=N}\|\widetilde{D}_{\mathcal{P}}\|_{p,q}^q=p\,
			B(q,1+2q/p). \end{equation} 
	\end{cor}

	\begin{rem}
		Observe that all norms we considered in this section are defined
		by integrals, where the integrands include the distribution of
		the discrepancy functions. Let $\cP$ be an $N$-element point set
		with $\cP\neq \Gamma_N$. Then, since the function
			$\alpha\mapsto \mathbb P(|D_{\mathcal P}|\geq \alpha)$
		is continuous, Theorem~\ref{theo1} yields that there is an
		interval
		$I$  of positive length such that $\mathbb P(|D_{\mathcal{P}}| \geq \alpha) > \mathbb
		P(|D_{\Gamma_N}|\geq \alpha)$ for all $\alpha \in I$. Therefore, equality in
	~\eqref{inf1} and~\eqref{inf3} holds only for $\mathcal{P}=\Gamma_N$. 
	With an analogue argumentation and referring to Theorem~\ref{theo2}
	instead of Theorem \ref{theo1}, we find that equality in
	~\eqref{inf2} and~\eqref{inf4} holds only for $\mathcal{P}=\Gamma_N^{\delta}$ for any $\delta\in[0,1/N)$.
	\end{rem}

	\section{Outlook}
	Considering the assertions of  Theorems \ref{theo1} and \ref{theo2},
	one might wonder whether the fact that for any number of points $N$ there exists a point set
	$\mathcal P'$ with $|\mathcal P'|=N$ so that for all points sets
	$\mathcal P$ with $|\mathcal P|=N$,
	 \[
		 \mathbb P(|D_{\mathcal{P}}| < \alpha) \leq \mathbb
		 P(|D_{\mathcal P'}| < \alpha),\qquad \alpha>0,
	 \]
	 extends to higher dimensions $d\geq 2$. Numerical calculations suggest
	 that such a general result is not true for $d= 2$ and $N\geq 4$.
	 Moreover, for $d\geq 3$ and even $N=1$ and $N=2$ such a general result  
	 is not true. This was proved for $N=1$ in \cite{Pillards2006} and for
	 $N=2$ in \cite{LarcherPillichshammer2007},
	 by showing 
	 that the unique $N$-element point sets that minimize $L_2$
	 discrepancy and
	 star discrepancy are different from each other.

	A different problem as posed above would be the following: for $d\geq 2$ and any
	non-negative integer $N$ find  
	 a function $f_N$  with the properties:
	\begin{enumerate} 
		\item It satisfies the inequality  
	 \[
		 \mathbb P(|D_{\mathcal{P}}| < \alpha) \leq 
		 f_N(\alpha),\qquad \alpha>0,
	 \]
for all point sets $\mathcal P$ with $|\mathcal P|=N$.

 \item It allows us to give sharp lower bounds for the
	 (quasi-)norm
	 of the discrepancy function in certain function spaces, for instance
	 $L_1$ or $L_p$, $p<1$.
	 \end{enumerate}
	 The problem of finding such functions $f_N$, at least in special cases,
	 will be investigated in the future.

	\subsection*{Acknowledgments}  R. Kritzinger is supported by the Austrian Science Fund (FWF),
Project F5509-N26, and M. Passenbrunner is supported by the Austrian Science Fund (FWF),
Project F5513-N26. Both projects are a part of the Special Research Program ``Quasi-Monte Carlo Methods: 
Theory and Applications''.

\bibliographystyle{plain}
\bibliography{distr}

\end{document}